\documentclass[11pt,a4paper]{amsart}
\usepackage{amsmath,caption,booktabs,lipsum}
\numberwithin{equation}{section}
\usepackage{amsthm}
\usepackage{pdfpages}
\usepackage{relsize}
\usepackage{amsfonts}
\usepackage{pdfpages}
\usepackage{amssymb}
\usepackage{mathrsfs}
\usepackage{tikz}
\usepackage{caption}
\usepackage{environ}
\usepackage{elocalloc}
\usepackage{url}
\usepackage{caption}
\usepackage{graphicx}
\usepackage[backend = bibtex, style=numeric-comp,doi=false,isbn=false,url=false]{biblatex}
\usepackage{CJKutf8}
\usepackage[normalem]{ulem}
\usepackage{xcolor}
\usepackage{etoolbox}
\usepackage{caption}
\usepackage{mathtools}
\usepackage{dcpic}
\usepackage{tikz-cd}
\usepackage[bottom]{footmisc}
\usepackage{hyperref}
\usepackage{enumitem}
\usepackage{stmaryrd}
\usetikzlibrary{positioning}
\usetikzlibrary{arrows,chains,positioning,scopes,quotes}
\usetikzlibrary{decorations.markings}
\allowdisplaybreaks

\newcommand{\ai}{\alpha}

\newcommand{\sing}{{\operatorname{Sing}}}

\newcommand{\be}{\beta}

\newcommand{\de}{\delta}
\newcommand{\De}{\Delta}

\newcommand{\lam}{\lambda}

\newcommand{\Om}{\Omega}
\newcommand{\si}{\sigma}
\newcommand{\Si}{\Sigma}

\newcommand{\s}{\subset}

\newcommand{\cp}{^\complement}

\newcommand{\la}{\langle}
\newcommand{\ra}{\rangle}
\newcommand{\ov}[1]{\overline{#1}}
\newcommand{\no}[1]{\left\lVert#1\right\rVert}

\DeclarePairedDelimiter{\ri}{\la}{\ra}

\newcommand{\m}{^{-1}}
\newcommand{\ts}{\otimes}

\newcommand{\pd}{\partial}

\newcommand{\na}{\nabla}

\newcommand{\ok}{\Om_{K}}
\newcommand{\oi}{\Om_{[\frac{3}{4}h^2,\infty)}}
\newcommand{\ot}{\Om_{(\frac{1}{2}h^2,\frac{3}{4}h^2)}}
\newcommand{\oo}{ \Om_{[0,\frac{1}{2}h^2]}}

\newcommand{\R}{\mathbb{R}}

\newcommand{\di}{\operatorname{div}}

\newcommand{\py}{{\phi'(|x|/\lam_y)^2}}
\newcommand{\pyy}{{\phi'(|x|/\lam_y)}}

\makeatletter
\def\thm@space@setup{%
	\thm@preskip=0.2cm plus 0cm minus 0cm
	\thm@postskip=\thm@preskip 
}
\makeatother
\theoremstyle{plain}
\newtheorem{thm}{Theorem}
\newtheorem{exam}{Example}[section]
\newtheorem{lem}[exam]{Lemma}

\newtheorem{fact}[exam]{Fact}

\newtheorem{defn}[exam]{Definition}

\theoremstyle{definition}
\newtheorem{conj}{Conjecture}

\title[On Simon's area-minimizing hypersurfaces with fractal singular sets]{On Simon's area-minimizing hypersurfaces with fractal singular sets}
\author{Zhenhua Liu}
\dedicatory{In memory of David Kraines}
\numberwithin{equation}{subsection}
\begin{document}
	\setlength{\abovedisplayskip}{5pt}
	\setlength{\belowdisplayskip}{5pt}
	\setlength{\abovedisplayshortskip}{5pt}
	\setlength{\belowdisplayshortskip}{5pt}
	\maketitle\vspace{-3em}
	\begin{abstract}
We verify that the stable minimal hypersurfaces with prescribed singular sets constructed by Simon in \cite{LSfr} are area-minimizing. Thus, Simon's work \cite{LSfr} settles the existence of area-minimizing hypersurfaces with fractal singular sets.
	\end{abstract}
	\section{Introduction}
Area-minimizing hypersurfaces of dimension $N$ are known to have singular sets of Hausdorff dimension at most $(N-7)$ \cite{HFts}. Leon Simon has proved that the singular sets are $(N-7)$ countably rectifiable \cite{LScy}. See also \cite{NV}. Frederick Almgren raised the following question \cite[Problem 5.4]{GMT}, 
\begin{conj}\label{ca}
		\emph{"Is it possible for the singular set of an area-minimizing integer (real) rectifiable current to be a Cantor type set with possibly non-integer Hausdorff dimension?"}
	\end{conj}
We prove that the ground-breaking work of Leon Simon \cite{LSfr} answered the above question in the hypersurface case. 
\begin{thm}\label{thmm}
The hypersurface $M\s\R^{n+m+l}$ in \cite[Theorem 3.7]{LSfr} is area-minimizing both as an integral current and as a mod $2$ current provided the construction parameters $\tau$ and $\de$ are small, the Lawson cone $M_{(n,m)}$ is area-minimizing, and $(n,m)\not=(3,5)$.
\end{thm}
Note that the area-minimizing Lawson cone $M_{(5,3)}$ is covered by the above theorem and $M_{(5,3)}$ and $M_{(3,5)}$ are equivalent up to ambient isometries, while Simon's hypersurface $M$ based on them are not equivalent. 

Compared to \cite{ZLa}, where the singularities are self-intersections of non-transverse immersions, the singularities in \cite{LSfr} are genuine and cannot be reduced into sum of smoothly embedded sheets.

Zhihan Wang has informed the author that GPT found proof for the $M_{(3,5)}$ case: the key lies in modifying the vector field in Fact \ref{fctpq} directly and no longer requires it to be the normalized gradient of a function. 
\subsection{Brief Recall of \cite{LSfr} and sketch of proof}
Let us first briefly recall the main results in \cite{LSfr}.

 The ambient manifold is (differential topologically)
\begin{align*}
\R^n\times \R^m\times\R^l,
\end{align*}
with $n\ge 3,m\ge 2,l\ge 1$ integers satisfying 
\begin{align*}
	n+m\ge 8.
\end{align*}
We use coordinate labels $$x=(x_1,\cdots,x_n)\in\R^n,\xi=(\xi_1,\cdots,\xi_m)\in\R^m,y=(y_1,\cdots,y_l)=\in\R^l.,$$

\begin{defn}\label{defnlc}
	The Lawson cone $(n,m)$, denoted by $M_{n,m}$, is defined in the standard flat $\R^n\times\R^m$ by
	\begin{align*}
		M_{n,m}=\left\{(x,y)\in\R^n\times\R^m:|\xi|=\sqrt{\frac{{m-1}}{{n-1}}}|x|\right\}.
	\end{align*}
\end{defn}
\begin{fact}\cite{BLep,PSmc}
	The Lawson cone $M_{n,m}$ is an area-minimizing integral current if and only if $n,m\ge 2,n+m\ge 9,$ or $(n,m)=(3,5),(4,4),(5,3).$
\end{fact}
Let $K$ be an arbitrary closed set of $\R^l$. The hypersurface $M$ in \cite[Theorem 3.7]{LSfr} is obtained by modifying $M_{n,m}\times\R^l$ and has singular set
\begin{align*}
	\sing M=\{0\}^n\times\{0\}^m\times K, 
\end{align*}Furthermore, $M$ is minimal and strictly stable in an altered metric
\begin{align*}
	g=dx^2+f(x,y)d\xi^2+dy^2.
\end{align*}
The hypersurface $M$ can be written as
\begin{align*}
	M=\left\{(x,\xi,y):|\xi|=u(x,y)\right\}.
\end{align*}
Here $f$ is smooth and $u$ is a modification of $\sqrt{\frac{{m-1}}{{n-1}}}|x|.$

Roughly speaking, in short distance to $K$, $M$ is a smooth minimal hypersurface with respect to the flat metric $\de=dx^2+d\xi^2+dy^2$. In mid distance to $K$, $M$ transitions into the cone $M_{(n,m)}\times \R^l$. In far distance to $K$, $M$ is the cone $M_{(n,m)}\times \R^l$.

We use the family of subcalibrations developed in \cite{ZLlc} to prove that $M$ is area-minimizing. 
\begin{defn}\label{defnh}
	Define
\begin{align*}
		E(x,\xi,y)=\begin{cases}
		\left(u(x,y)^2-|\xi|^2\right)u(x,y)^d,&\textnormal{if }|\xi|\le u(x,y),\\
		\left(u(x,y)^2-|\xi|^2\right)|\xi|^d,&\textnormal{if }|\xi|\ge u(x,y),\\
	\end{cases}
\end{align*}with $d\in(0,\infty)$.
\end{defn}
Note that the zero set of $E$ is $M.$ The divergence of the normalized gradient vector field $\frac{\na^g E}{\no{\na^g E}_g}$ of $E$ equals the signed mean curvature scalar of the level sets of $E.$ \begin{lem}\label{lemsub}
If $d=\frac{9}{5},(n,m)=(4,4),(5,3)$ or $d=3,n+m\ge 9$	then we have in the metric $g$\begin{align*}
		\di_g \frac{\na^g E}{\no{\na^g E}_g}\begin{cases}
			\ge 0,&\textnormal{if }u(x,y)-|\xi|\ge 0,\\
			\le0,&\textnormal{if } u(x,y)-|\xi|\le 0,\\
		\end{cases}
	\end{align*}
\end{lem}
Applying divergence theorem as in \cite{GDEPss,DPFMsi} shows  that $M$ is area-minimizing.

The proof of Lemma \ref{lemsub} is straightforward calculation. The very special form of $E$ in Definition \ref{defnh} allows us to simplify $\frac{\na^g E}{\no{\na^g E}_g}$ significantly and factorize 
\begin{align}\label{eqfact}
(\textnormal{Positive factors})\di_g\frac{\na^g E}{\no{\na^g E}_g}=&(u^2-|\xi|^2)\textnormal{Quadratic form in variables }u^2\textnormal{ and }|\xi|^2
\end{align}
The quadratic form in $u,|\xi|$ has coefficients in $|Du|^2$ and $u\De u$. In mid to far distance from $K$, the metric $g$ is quantitatively close to the flat metric $\de$ and the hypersurface $M$ is quantitatively close to the cone $M_{(n,m)}\times \R^l$, so the calculations are very similar to those in \cite{ZLlc} and we are reduced to proving positiveness of real coefficient quadratic forms. Near $K,$ we need refined estimates from \cite{LSfr} by using  $|Du|^2 $ to control $u\De u$. In the easier cases, $u\De u\le \frac{m-1}{n}|Du|^2$ suffices to finish the proof. In general, the quadratic form in $u,|\xi|$ has linear over linear rational coefficients in $|Du^2|$, and a Cauchy-Schwarz argument finishes the proof.
\subsection{Overview of the paper}
Our paper will be structured as follows. 

In Section \ref{secprelim}, we will fix notations and collect several preliminaries.

In Section \ref{seces}, we will calculate key estimates based on \cite{LSfr}.

In Section \ref{seccalc}, we will calculate $\di_g \frac{\na^g F}{\no{\na^g F}_g}$.

In Section \ref{secfact+}, we will prove the factorization (\ref{eqfact}) and positivity in case of $u\ge v.$

In Section \ref{secfact-}, we will prove the factorization (\ref{eqfact}) and positivity in case of $v\ge u.$

In Section \ref{secwrap}, we will finish the proof of Theorem \ref{thmm}.

In Section \ref{secconc} we will give some discussions.
\section*{Remembering David Kraines}
Professor Kraines played a crucial role in my mathematical development
and in my life. He first contacted me after I was admitted to Duke,
and we spoke on the very day I arrived in Durham. His encouragement
to pursue mathematics meant an extraordinary amount to a young person
newly arrived in an unfamiliar country. From the beginning, he made
me feel welcome and believed in what I might become.

He gave me exceptional opportunities, welcoming me into Duke's
PRUV undergraduate research fellowship in both my freshman and
sophomore years. He chose Professor Hubert Bray as my research advisor, a
decision that profoundly influenced my mathematical development and
set me on the path toward geometry. Professor Bray introduced me to
the paper on subcalibrations that helped shape the present work.
In this sense, this paper, too, owes something to Professor Kraines' guidance.

His generosity extended well beyond mathematics. He regularly
invited students to monthly lunches at his home, bringing us together
with warmth and care. He also introduced me to my future wife. It was through him we came to know each other.

The last time I saw Professor Kraines was when I graduated from Duke. I had
hoped to visit him this May, but severe flight delays and subsequent academic travels forced me to
cancel the trip. I wrote to him again during the summer and received
no reply. Only later did I learn that he had passed away around the
time I sent that email. I deeply regret missing the opportunity to
see him again and to tell him how much his encouragement, generosity,
and friendship had meant to me.

I dedicate this paper to his memory, with lasting gratitude for the
opportunities he gave me and the people he brought into my life.
\section*{Acknowledgements}The author is deeply grateful to Professor Leon Simon for his extraordinary generosity and sustained support. His direct encouragement to complete this work was instrumental in bringing the paper to fruition. This paper is also a tribute to Professor Simon’s pioneering mathematical vision. The author extends his sincere thanks to Professors William Allard, Robert Bryant, Simon Brendle, and Frank Morgan for their generous support, and to Professor Hubert Bray for his hospitality and for bringing the subcalibration method back to the author’s attention. Last but not least, I would also like to thank Xunjing Wei for constant support.
\section{Preliminaries}\label{secprelim}
Let us first set up notations and then collect several standard facts.
\subsection{Notation conventions}
Riemannian quantities and operators in metric $g$ will be denoted with superscript or subscript $g$. A Riemannian quantity or operator without label $g$ will always mean the quantity or operator taken in the standard flat metric on $\R^n\times\R^m\times\R^l$. 

Lower case Roman letters will be reserved for the coordinate labels $(x,y)$ and lower case Greek letters will be reserved for the coordinate label $\xi.$	Partial differentiation with respect to $(x,\xi,y)$ coordinate will be the symbol $D.$ Lower case subscript means partial differentiation with respect to other variables.

The symbol $v$ will be reserved to denote $|\xi|,$ i.e.,
\begin{align*}
	v=|\xi|.
\end{align*}
The symbol $r$ will be reserved to denote $|x|,$ i.e., $$r=|x|.$$
 The symbol $\ai_0$ will always mean $$\ai_0=\sqrt{\frac{m-1}{n-1}}.$$
\subsection{Quadratic polynomials}Let us state two handy facts about quadratic polynomials.
\begin{fact}\label{fctl}
	For a quadratic polynomial $L:\R\to\R$ let $L_2,L_1,L_0$ be the coefficients of second, first and zeroth order terms respectively. If 
	\begin{align}\label{eqlll}
		L_2+L_1+L_0,L_1+2L_0,L_0\ge0,
	\end{align}then we have
	\begin{align}\label{eqlllt}
		\inf_{t\in[0,1]}L(t)\ge0.
	\end{align}Furthermore, if strict inequality holds in (\ref{eqlll}), then strict inequality holds in (\ref{eqlllt})
\end{fact}
\begin{proof}
	We have\begin{align*}
		L(t)=L_2t^2+L_1t+L_0=(L_2+L_1+L_0)t^2+t(1-t)(L_1+2L_0)+(1-t)^2L_0.
	\end{align*}The claimed fact follows directly.
\end{proof}
\begin{fact}\label{fctpoly}
	If $P_0:\R\to\R $ is a quadratic polynomial with strictly positive infimum on $[0,1]$ and $Q_0$ is another quadratic polynomial, then $P_0+\tau Q_0$ has strictly positive infimum on $[0,1]$ provided $\tau$ is small.
\end{fact}
\begin{proof}
	By positive infimum of $P_0,$ $\frac{|Q_0|}{P_0}$ is a bounded continuous function on $[0,1]$. We have
	\begin{align*}
		P_0+\tau Q_0=P_0\left(1+\tau\frac{Q_0}{P_0}\right)\ge P_0\left(1-\tau\frac{|Q_0|}{P_0}\right)>0.
	\end{align*}
\end{proof}
\subsection{Subcalibrations} 
Let us introduce subcalibrations.
\begin{fact}\label{fctsubc}
Let $O$ be an oriented open set in $\R^{N+1}$ and assume that $O$ represents an integral current and let $M=\pd O$ both in the sense of sets and currents. If $E$ is a  function on $\R^{N+1}$ satisfying,
	\begin{itemize}
		\item $E\m(0)=M,E\m(-\infty,0)=O,$
		\item $\frac{\na^g E}{\no{\na^g E}_g}$ is bounded, continuous outside a set $\sing \frac{\na^g E}{\no{\na^g E}_g}$ of Hausdorff $N$-dimensional measure $0$,
		\item $\frac{\na^g E}{\no{\na^g E}_g}$ is locally Lipschitz outside a set $\sing^{Lip} \frac{\na^g E}{\no{\na^g E}_g}$ where the restriction of Hausdorff $N$-dimensional measure is $\si$-finite,
		\item $\di_g \frac{\na^g E}{\no{\na^g E}_g}$ has the same signs with $E$ off $M,$ i.e.,.
		\begin{align*}
			\di_g \frac{\na^g E}{\no{\na^g E}_g}
			\begin{cases}
				\ge 0, &\textnormal{ on }O\cp\\
				\le 0,& \textnormal{ on }O.
		\end{cases}	\end{align*}
	\end{itemize}
	Then $M$ is area-minimizing both as an integral current and as a mod $2$ current. 
\end{fact}
We call $E$ a subcalibration profile of $M.$
\begin{proof} Let us first prove that $O$ is a perimeter minimizing Caccioppoli set.The proof is the same as \cite[Proof of Proposition 4.1]{DPFMsi}. In \cite{DPFMsi}, the authors assume that $\frac{\na^g E}{\no{\na^g E}_g}$ is $W^{1,1}_{loc}$ in order to use divergence theorem on sets of finite perimeter. However, \cite[Theorem 3.10]{DPggdiv}  shows that the second and the third bullet in the lemma suffice for application of divergence theorem.
	
	Next by \cite[Proposition 4.1]{LPgp}, we deduced that $M$ is an area-minimizing integral current. By \cite[Proposition 2.1, $v=2$]{BWmod4}, any compactly supported $(N+1)$-dimensional mod $2$ current can be lifted to an integral chain that represents an oriented bounded measurable set $O'$ with finite perimeter. Applying \cite[Proposition 3.2, $v=,T_1=O,T_2=O'$]{BWmod4}, we deduce that $M$ is area-minimizing mod $2.$
\end{proof}
\section{Key estimates from \cite{LSfr}}\label{seces}
In this section, we will derive several key estimates based on \cite{LSfr}.

Recall \cite[Equation set 3.1]{LSfr} that the smooth auxiliary function $h:\R^l\times \R_{\ge0}$ has $K$ as its zero set, i.e.,
\begin{align*}
	h>0\textnormal{ on }U=\R^l\setminus K, h=0\textnormal{ on }K.
\end{align*}The $C^3$ norm of $h$ is controlled by construction parameter $\tau$.

We will divide $\R^n\times\R^m\times \R^l$ into four disjoint subsets
\begin{align*}
	\R^n\times\R^m\times \R^l=\Om_{K}\cup \Om_{[\frac{3}{4}h^2,\infty)}\cup\Om_{(\frac{1}{2}h^2,\frac{3}{4}h^2)}\cup \Om_{[0,\frac{1}{2}h^2]}.
\end{align*}
\begin{defn}
	Define
	\begin{align*}
		\ok=&\left\{(x,\xi,y):y\in K\right\},\\
		\oi=&\left\{(x,\xi,y):|x|\ge\frac{3}{4}h^2(y),y\in U\right\},\\
		\ot=&\left\{(x,\xi,y):\frac{3}{4}h^2(y)>|x|>\frac{1}{2}h^2(y),y\in U\right\},\\
		\oo=&\left\{(x,\xi,y):\frac{1}{2}h^2(y)\ge|x|,y\in U\right\}.
	\end{align*}
\end{defn}
The special cases of $f,u$ are as follows.
\begin{fact}\cite[Equation 3.5 to Theorem 3.7]{LSfr}
	We have
	\begin{align*}
		u=\begin{cases}
			\ai_0r,&\textnormal{ on }\ok,\\
		\ai_0r,&\textnormal{ on }\oi,	\end{cases}
	\end{align*}and
		\begin{align*}
		f=\begin{cases}
			1,&\textnormal{ on }\ok,\\
			1,&\textnormal{ on }\oo.	\end{cases}
	\end{align*}
\end{fact}
The needed estimates on $u,f$ can be separated into to cases.
\subsection{Case $1$: almost conical}
In $\ok\cup \oi\cup\ot,$ we need the following estimates.
\begin{fact}\label{fctfu}
On $\ok\cup\oi\cup\ot$, we have
\begin{align*}
	&f=1+O(\tau),Df=O(\tau),	Du=\ai_0Dr+O(\tau),D^2u=\ai_0D^2r+O(\tau),\\&|Du|^2=\ai_0^2+O(\tau),\De u=\ai_0\frac{n-1}{r}+O(\tau)\\&\frac{u}{\ai_0r}=1+O(\tau),D^2u(\na u,\na u)=O(\tau)u\m,\\&uDf\cdot Du=O(\tau),u\De u=(m-1)+O(\tau).
\end{align*} 
\end{fact}
Here $r$ is Simon's notation for $|x|,$ i.e., $r=|x|$ and $\ai_0$ denotes $\sqrt{\frac{m-1}{n-1}}$ \cite[just after equation 3.2]{LSfr}.
\begin{proof}On $\ok,$ we have $u=\ai_0|\xi|$ \cite[equation 3.5]{LSfr} and $f=1$ \cite[Theorem 3.7]{LSfr}. The estimates follow directly by setting $\tau=0$, and $\De r=\frac{n-1}{r}$, $D^2(\na r,\na r)=0.$

On $\oi\cup\ot$, by \cite[Equation 3.6 and Theorem 3.7]{LSfr}, we get the first two lines of estimates. For the second line of estimate, note that $u=\ai_0r$ on $\oi$ and we have on $\ot$
\begin{align*}
\left|\frac{u}{\ai_0r}-1\right|\le&\frac{Ch^4}{r}\le\frac{Ch^4}{\frac{1}{2}h^2}=2Ch^2=O(\tau),\\		\sum_{1\le i,j\le n+l}D_iuD_juD_{ij}u=&(D_ir+O(\tau))(D_jr+O(\tau))(D_{ij}r+O(\tau))\\=&\sum_{1\le i,j\le n+l}D_irD_jrD_{ij}r+O(\tau)(|Du|^2+|Du||D^2u|)\\=&O(\tau)(1+O(D^2r))=O(\tau)r\m\\=&O(\tau)u\m.\\
	\end{align*}
For the last line of estimate, on $\ot$, $u=O(\tau),Du=O(1),Df=O(\tau)$ so we deduce $uDf\cdot Du=O(\tau^2)$ and $u\De u=(m-1)+O(\tau).$ On $\oi,$ recall that the Euler-Lagrange equation for $u,f$ is \cite[Proof of Theorem 3.7 just before equation (3)]{LSfr}
\begin{align*}
	\frac{1}{2}\left(m+(1+f|Du|^2)\m\right)Df\cdot Du=-f\De u+\frac{f^2}{1+f|Du|^2}D^2u(\na u,\na u)+\frac{m-1}{u}.
\end{align*}
Substitute $u=\ai_0r$ gives $u\De u=m-1$ and
\begin{align*}
	&\frac{m+(1+f\ai_0^2)\m}{2}uDf\cdot Du\\
	=&-uf\frac{m-1}{u}+\frac{f^2}{1+f\ai_0^2}0+(m-1)\\
=&(m-1)(1-f)\\
=&O(\tau).
\end{align*}Dividing multiplying both sides by $\frac{1}{\frac{m+(1+f\ai_0^2)\m}{2}}\le \frac{2}{m}$ finishes the estimate for $uDf\cdot Du.$
\end{proof}
\subsection{Case $2$: approximation via $\phi$}
On $\oo$, we need to control $u\De u-(m-1).$ In this subsection, we always assume that $(x,\xi,y)\in\oo.$ 
	The Euler-Lagrange equation for $u$ \cite[between (2) and (3) in proof of Theorem 3.7]{LSfr} gives 
\begin{align}\label{eqel}
u\De u-(m-1)=\frac{u}{1+(Du\cdot Du)}D^2u(\na u,\na u)
\end{align}
Our goal is to prove an estimate roughly of the form (Lemma \ref{lemudum})
\begin{align*}
			u\De u-(m-1)\le R(|Du|)|Du|^2+O(\de)|Du|^2,
\end{align*}with $R$ a linear-over-linear rational function.

The comparison estimates in \cite[Theorem 6.3]{LSfr} allow us to use the radially symmetric solution $\phi$, i.e., $|\xi|=\phi(|x|),$ corresponding to the Hardt-Simon foliation \cite{HS}, to estimate $u.$ Let us first recall some facts about the function $\phi:[0,\infty)\to[1,\infty)$.
\begin{fact}\label{fctphi}
	We have a smooth solution $\phi:[0,\infty)\to[0,\infty)$ to the ordinary differential equation,
	\begin{align*}
	\frac{\phi''\phi}{1+(\phi')^2}+\frac{n-1}{r}\phi\phi'={m-1},
	\end{align*}
	satisfying
	\begin{align*}
		&\phi(0)=1,\phi'(0)=0,\phi''(0)=\frac{m-1}{n},\phi'''(0)=0,\phi^{(4)}(0)=\frac{3 (m-1)^2 \left(2 m-n^2-2\right)}{n^3 (n+2)}.\\
		&0<\phi'<\ai_0,\phi''>0,\\
		&\phi-r\phi'>0,\max\{1,\ai_0r\}<\phi<1+\ai_0r.
	\end{align*}
\end{fact}
\begin{proof}
	Everything except for $\phi''(0)$ and $\phi^{(4)}(0)$ is direct restatement of \cite[from equation 4.2 to equation 4.6, equation set 4.10]{LSfr}. For $\phi''(0),$	note that $\lim_{r\to0}\frac{\phi'(r)}{r}=\phi''(0).$ Thus Fact \ref{fctphi} gives
	\begin{align*}
		\phi''(0)+(n-1)\phi''(0)=(m-1),
	\end{align*}	i.e.,
	\begin{align*}
		\phi''(0)=\frac{m-1}{n}.
	\end{align*}
	For $\phi'''(0)$ and $\phi^{(4)}(0)$, use \cite[integral equation bewtween 4.2 and 4.3]{LSfr}
	\begin{align*}
		&\frac{\phi'}{\sqrt{1+(\phi')^2}}=\frac{1}{r^{n-1}}\int_0^r\frac{m-1}{\phi(s)\sqrt{1+(\phi'(s))^2}}s^{n-1}ds\\
		=&\frac{1}{r^{n-1}}\int_0^r\frac{m-1}{\left(1+\frac{1}{2}\phi''(0)s^2+O(s^3)\right)\sqrt{1+(\phi''(0)s+O(s^2))^2}}s^{n-1}ds\\
		=&\frac{1}{r^{n-1}}\int_0^r(m-1)(1-\frac{1}{2}\phi''(0)s^2+O(s^3))\left(1-\frac{1}{2}(\phi''(0))^2s^2+O(s^3)\right)s^{n-1}ds\\
		=&\frac{1}{r^{n-1}}\int_0^r(m-1)\left(1-\phi''(0)s^2+O(s^3)\right)s^{n-1}ds\\
		=&(m-1)\left(\frac{1}{n}r-\frac{\phi''(0)(1+\phi''(0))}{2(n+2)}r^3+O(r^4)\right)\\
		=&\phi''(0)r-\frac{m-1}{2(n+2)}\phi''(0)(1+\phi''(0))r^3+O(r^4).
	\end{align*}
		By Taylor's theorem we have
	\begin{align*}
		&\sqrt{1+(\phi'(r))^2}=\sqrt{1+\left(\phi''(0)r+\frac{1}{2}\phi'''(0)r^2+O(r^3)\right)^2}\\
		=&\sqrt{1+(\phi''(0))^2r^2+\phi''(0)\phi'''(0)r^3+O(r^4)}\\
		=&1+\frac{1}{2}(\phi''(0))^2r^2+\frac{1}{2}\phi''(0)\phi'''(0)r^3+O(r^4)
	\end{align*}
	and 
	\begin{align*}
		&\frac{\phi'}{\sqrt{1+(\phi')^2}}\\=&\left(\phi''(0)r+\frac{1}{2}\phi'''(0)r^2+\frac{1}{6}\phi^{(4)}(0)r^3+O(r^4)\right)\left(1-\frac{1}{2}(\phi''(0))^2r^2-\frac{1}{2}\phi''(0)\phi'''(0)r^3+O(r^4)\right)\\
		=&\phi''(0)r+\frac{1}{2}\phi'''(0)r^2+\left(-\frac{1}{2}(\phi''(0))^3+\frac{1}{6}\phi^{(4)}(0)\right)r^3+O(r^4).
	\end{align*}
	Thus, we must have $\phi'''(0)=0$ and 
	\begin{align*}
		\phi^{(4)}(0)=6\left(-\frac{m-1}{2(n+2)}(1+\phi''(0))+\frac{1}{2}(\phi''(0))^2\right)\phi''(0)=\frac{3 (m-1)^2 \left(2 m-n^2-2\right)}{n^3 (n+2)}.
	\end{align*}
\end{proof}
There is a positive parameter $\de$ in \cite[Theorem 6.3]{LSfr} that can be taken arbitrarily small. However, we cannot set $\de=0$ as only positive $\delta$ gives positive $\tau$ and thus existence of $u$ that works for the construction of $M$ in \cite[Theorem 3.7]{LSfr}. Therefore, in the following estimates $\de$ can be taken to be any positive real value, but not $0$ and this is not a mistake.

Let us first reduce the estimates
\begin{lem}\label{lemuduu}
	We have
	\begin{align*}
		u\De u-(m-1)\le \frac{\phi(|x|/\lam_y)\phi''(|x|/\lam_y)|}{1+{\phi'(|x|/\lam_y)}^2}|Du|^2+O(\de)|Du^2|,
		\end{align*}\end{lem}Here big $O$ depends only on $m,n.$
\begin{proof}
Let $$\lam_y=u(0,y),\phi_{\lam_y}(r)=\lam_y\phi(r/\lam_y),
$$	\cite[Theorem 6.3(i)]{LSfr} gives
	\begin{align*}	&(\lam_y+|x|)\m|u(x,y)-\phi_{\lam_y}(|x|)|<\de,\\
	&|D_{x,z}(u(x,z)-\phi_{\lam_y}(|x|))|_{z=y}<\de,\\
	&(\lam_y+|x|)|D_{x,z}^2(u(x,z)-\phi_{\lam_y}(|x|))|_{z=y}<\de.	\end{align*}
Let us unpack the quantities. Let $\nu\in\R^{n+m+l}$ be a vector and $\pi_x\nu$ be its projection to the $x$-components.
\begin{align*}
	&D_{x,z}\phi_{\lam_y}(|x|)|_{z=y}=\lam_yD_{x,z}\phi(|x|/\lam_y)|_{z=y}=\phi'(|x|/\lam_y)D|x|,\\
	&|D_{x,z}\phi_{\lam_y}(|x|)|_{z=y}={\phi'(|x|/\lam_y)}<\ai_0,\\
	&D_{x,z}^2\phi_{\lam_y}(|x|)|_{z=y}=D_{x,z}(\phi'(|x|/\lam_y)D|x|)|_{z=y}=\lam_y\m\phi''(|x|/\lam_y)D|x|\ts D|x|+\phi'(|x|/\lam_y)D^2|x|,\\
	&D|x|\ts D|x|(\nu,\nu)=|\ri{\nu,\na|x|}|^2,D^2|x|(\nu,\nu)=\frac{|\pi_x\nu|^2}{|x|}-\frac{|\ri{\nu,x}|^2}{|x|^3}=\frac{|\pi_x\nu|^2-|\ri{\nu,\na x}|^2}{|x|}.
\end{align*}		
Note that $u$ in the $x$-directions only depend on $|x|$ \cite[equation 3.5]{LSfr}, thus, $D^2|x|(\na u,\na u)=0.$ Now apply Fact \ref{fctphi} to the above calculation gives
	\begin{align*}
		&\left|\frac{u(x,y)}{\phi_{\lam_y}}-1\right|<\de \frac{\lam_y+r}{\lam_y\phi(r/\lam_y)}<\de \frac{\lam_y+r}{\lam_y\max\{1,\ai_0r/\lam_y\}}\le\de\left(1+\frac{1}{\ai_0}\right)	\end{align*}
and
\begin{align}
|Du|&\le |D_{x,z}(u-\phi_{\lam_y}(|x|))|_{z=y}+|D_{x,z}\phi_{\lam_y}(|x|)|_{z=y}\label{eqdu1}
\le  \de +\ai_0,\\
||Du|^2-|\phi'(|x|/\lam_y)|^2|&=\left||Du|^2-|D_{x,z}\phi_{\lam_y}(|x|)|_{z=y}^2\right|\\&\le\left|D_{x,z}(u-\phi_{\lam_y}(|x|))|_{z=y}\right|\left(|Du|+|D_{x,z}\phi_{\lam_y}(|x|)|_{z=y}\right)\\&\le\de(2\ai_0+\de).\label{eqdu2}
	\end{align}
	thus,
\begin{align*}
&\left|\frac{1+{\phi'(|x|/\lam_y)}^2}{1+|Du|^2}-1\right|=\left|\frac{1+|D_{x,z}\phi_{\lam_y}(|x|)|_{z=y}^2}{1+|Du|^2}-1\right|\\=&\frac{\left||Du|^2-|D_{x,z}\phi_{\lam_y}(|x|)|_{z=y}^2\right|}{1+|Du|^2}\\\le&\frac{\left|D_{x,z}(u-\phi_{\lam_y}(|x|))|_{z=y}\right|\left(|Du|+|D_{x,z}\phi_{\lam_y}(|x|)|_{z=y}\right)}{1+|Du|^2}\\
	<&\de(2\ai_0+\de),
\end{align*}
and finally
\begin{align*}
	&	\left|
\phi_{\lam_y}	D^2u(\na u,\na u)-\phi(|x|/\lam_y)\phi''(|x|/\lam_y)|\ri{\na u,\na|x|}|^2\right|\\
	=&	\left|
	\phi_{\lam_y}D^2u(\na u,\na u)-\phi_{\lam_y}(|x|)D_{x,z}^2\phi_{\lam_y}(|x|)\ri{\na u,\na|x|}|^2|_{z=y}\right|\\
	=&\phi_{\lam_y}\left|D_{x,z}^2(u-\phi_{\lam_y}(|x|))(\na u,\na u)|_{z=y}\right|\\
	<&\lam_y(1+\ai_0|x|/\lam_y)\left|D_{x,z}^2(u-\phi_{\lam_y}(|x|))(\na u,\na u)|_{z=y}\right|\\
	\le&\max\{1,\ai_0\}(\lam_y+|x|)\left|D_{x,z}^2(u-\phi_{\lam_y}(|x|))(\na u,\na u)|_{z=y}\right|\\
	<&\max\{1,\ai_0\}\de|\na u|^2.
\end{align*}
On the other hand, by \cite[Equation 4.3]{LSfr}, we have
\begin{align*}
	\frac{\phi''\phi}{1+(\phi')^2}+\frac{n-1}{r}\phi\phi'={m-1}.
\end{align*}
This implies that $\frac{\phi''\phi}{1+(\phi')^2}\le m-1.$
Thus, we can evaluate
\begin{align*}
&	\frac{u}{1+(Du\cdot Du)}D^2u(\na u,\na u)\\
=&\frac{u}{\phi_{\lam_y}}\frac{1+{\phi'(|x|/\lam_y)}^2}{1+|Du|^2}\frac{\phi_{\lam_y}D^2u(\na u,\na u)}{1+{\phi'(|x|/\lam_y)}^2}\\
=&(1+O(\de))(1+O(\de))\frac{\phi(|x|/\lam_y)\phi''(|x|/\lam_y)|\ri{\na u,\na|x|}|^2+O(\de)|\na u|^2}{1+{\phi'(|x|/\lam_y)}^2}\\
\le&(1+O(\de))\frac{\phi(|x|/\lam_y)\phi''(|x|/\lam_y)|+O(\de)}{1+{\phi'(|x|/\lam_y)}^2}|\na u|^2\\
=&\left(\frac{\phi(|x|/\lam_y)\phi''(|x|/\lam_y)|}{1+{\phi'(|x|/\lam_y)}^2}+O(\de)\right)(1+O(\de))|Du|^2\\=&\frac{\phi(|x|/\lam_y)\phi''(|x|/\lam_y)|}{1+{\phi'(|x|/\lam_y)}^2}|Du|^2+O(\de)|Du^2|.
\end{align*}Here big $O$ depends only on $m,n.$
\end{proof}
Lemma \ref{lemuduu} tells us we need to estimate $\frac{\phi''\phi}{1+(\phi')^2}.$ The idea is that we  want that the estimate, after comparison, depends roughly only on $|Du|^2$. Equivalently, we want to bound $\frac{\phi''\phi}{1+(\phi')^2}$ using functions of $(\phi')^2$.
\begin{lem}\label{lemphip}
	We have
	\begin{align*}
		\frac{\phi''\phi}{1+(\phi')^2}\le\frac{2(m-1)\left(\ai_0^2-(\phi')^2\right)}{2n\ai_0^2+(m-n-1)(\phi')^2}
	\end{align*}
\end{lem}
\begin{proof}	
Set 
\begin{align*}
	\Phi=&\frac{\phi''\phi}{1+(\phi')^2},\Si=(\phi')^2
	,\Psi(\Si)=\frac{2(m-1)(\ai_0^2-\Si)}{2n\ai_0^2+(m-n-1)\Si}.
\end{align*}
Recall Fact \ref{fctphi}. Since $\phi'(r)>0$ for $r>0$ and $\phi''>0$ we can regard $\Si=(\phi')^2$ as a variable. Then our goal is to prove that $\Phi(\Si)-\Psi(\Si)\le 0$  for $\Si\ge 0.$ Let us first see what happens at $\Si=0,$ i.e., $r=0.$ We have
\begin{align*}
	\Phi(\Si)|_{\Si=0}&=\phi''(0)=\frac{m-1}{n},\\
	\Psi(\Si)|_{\Si=0}&=\frac{2(m-1)^2}{2n(m-1)}=\frac{m-1}{n}.
\end{align*}
Thus, $$	\Phi(\Si)|_{\Si=0}-\Psi(\Si)|_{\Si=0}=0.$$ Furthermore, we have
\begin{align}
	&\frac{d\Phi}{d\Si}\bigg|_{\Si=0}=\lim_{\Si\to 0}\frac{\Phi(\Si)-\Phi(0)}{\Si}\\
	=&\lim_{r\to 0}\frac{\frac{\phi(r)\phi''(r)}{1+(\phi'(r))^2}-\phi''(0)}{(\phi'(r))^2}\\
	=&\lim_{r\to 0}\frac{\left(1+\frac{1}{2}\phi''(0)r^2+O(r^3)\right)\left(\phi''(0)+\frac{1}{2}\phi^{(4)}(0)r^2+O(r^3)\right)\left(1-(\phi''(0))^2r^2+O(r^3)\right)-\phi''(0)}{(\phi''(0))^2r^2+o(r^3)}\\
	=&\frac{\phi^{(4)}(0)+(\phi''(0)^2)-2(\phi''(0))^3}{2(\phi''(0))^2}\\
	=&-\frac{(n-1) (m+n-1)}{n (n+2)}.\label{eqdphi0}
\end{align}
On the other hand, we have
\begin{align}
	\frac{d\Psi}{d\Si}=-\frac{2 (m-1)^2 (n-1) (m+n-1)}{((n-1) \Si (-m+n+1)-2
		(m-1) n)^2}.
\end{align}
At $\Si=0,$ we have
\begin{align}\label{eqdpsi0}
		\frac{d\Psi}{d\Si}\bigg|_{\Si=0}=-\frac{2 (m-1)^2 (n-1) (m+n-1)}{(2
			(m-1) n)^2}=-\frac{(n-1)(m+n-1)}{2n^2}.
\end{align}
Combining (\ref{eqdphi0}) and (\ref{eqdpsi0}), we have
\begin{align*}
	\frac{d}{d\Si}(\Phi-\Psi)\bigg|_{\Si=0}=(n-1)(m+n-1)\left(\frac{1}{2n^2}-\frac{1}{n(n+2)}\right)<0.
\end{align*}Thus $\Phi(\Si)-\Psi(\Si)<0$ for $\Si$ close to $0.$ By \cite[p.89 I. Lemma]{WWode}, to prove that $\Phi(\Si)-\Psi(\Si)\le 0$ for all $\Si,$ it suffices to show that 
\begin{align*}
\frac{d}{d\Si}(	\Phi(\Si)-\Psi(\Si)
)<0,\end{align*}
if $\Phi=\Si.$

Recall \cite[Equation 4.3]{LSfr}:
\begin{align*}
	\frac{\phi''\phi}{1+(\phi')^2}+\frac{n-1}{r}\phi\phi'={m-1}.
\end{align*}
	Differentiating again gives
\begin{align*}
&\frac{d\Phi}{dr}=	\left(	\frac{\phi''\phi}{1+(\phi')^2}\right)'\\=&-\frac{n-1}{r}(\phi')^2-\frac{n-1}{r}\phi\phi''+\frac{n-1}{r^2}\phi\phi'\\
=&-\frac{n-1}{r}\Si-\frac{n-1}{r}(1+\Si)\Phi+\frac{m-1}{r}-\frac{1}{r}\Phi\\
=&-\frac{n-1}{r}\Si-\frac{n+(n-1)\Si}{r}\Phi+\frac{m-1}{r}.
\end{align*}
And
\begin{align*}
	\frac{d\Si}{dr}=2\phi'\phi''=2\phi'\frac{1+\Si}{\phi}\Phi=2\Si\frac{1+\Si}{\phi\phi'}\Phi=\frac{n-1}{r}\frac{2\Si(1+\Si)}{m-1-\Phi}\Phi
\end{align*}
Thus, we have
\begin{align*}
	\frac{d\Phi}{d\Si}=&\frac{\frac{d\Phi}{dr}}{\frac{d\Si}{dr}}=\frac{(m-1-\Phi)\left(-(n-1)\Si-\left(n+(n-1)\Si\right)\Phi+(m-1)\right)}{2(n-1)\Si(1+\Si)\Phi}.
\end{align*}
Direct calculation gives
\begin{align*}
&\frac{d}{d\Si}(\Phi-\Psi)|_{\Phi=\Psi}\\=&\frac{(n-1) (m+n-1) }{4
	(\Si+1) ((n-1) \Si (m-n-1)+2 (m-1) n)^2}\\&\cdot\left(-2 (m-1) \Si \left(m (2
	n-5)-2 n^2+n+5\right)-(n-1) \Si^2
	\left((m-n)^2-1\right)-4 (m-1)^2 (n-2)\right).
\end{align*}
The first factor is positive. We need to show that the second factor is negative. Set $\Si=\frac{m-1}{n-1}t.$ Then $t\in[0,1]$. The second factor is
\begin{align*}
	&\left(-2 (m-1) \Si \left(m (2
	n-5)-2 n^2+n+5\right)-(n-1) \Si^2
	\left((m-n)^2-1\right)-4 (m-1)^2 (n-2)\right)\\=&\frac{(m-1)^2 \left(t^2 \left(2 m n+(1-m)
		m-m-n^2+1\right)+t \left(-4 m n-10 (1-m)+4 n^2-2
		n\right)-4 n^2+12 n-8\right)}{n-1}.
\end{align*}
Set $$L(t)=t^2 \left(2 m n+(1-m)
m-m-n^2+1\right)+t \left(-4 m n-10 (1-m)+4 n^2-2
n\right)-4 n^2+12 n-8,$$ with $L_2,L_1,L_0$ be the coefficients of $t^2,t,t^0$ respectively. 
By Fact \ref{fctl}, it suffices to show that 
\begin{align*}
	L_2+L_1+L_0,L_1+2L_0,L_0<0.
\end{align*}We have
\begin{align*}
	L_0=&-4n^2+12n-8=-4\left(n-\frac{3}{2}\right)^2+1<0,
\end{align*}for $n\ge 3.$
For $n\ge 3,$ we have
\begin{align*}
	L_1+2L_0=&m (10-4 n)+(-4 n+22 )n-26\\=&(10-4n)(m+n-8)-20n+54\\
	\le&-6\\<&0.
\end{align*}
Finally, we have
\begin{align*}
	L_2+L_1+L_0=-17-(-10+(m+n))(m+n)\le -1<0.
\end{align*}
\end{proof}
\begin{lem}\label{lemudum}
For
\begin{align*}
	R(s)=\frac{2(m-1)(\ai_0^2-s^2)}{2n(\ai_0^2-s^2)+(m+n-1)s^2}.
\end{align*}	we have	\begin{align}\label{eqfl}
u\De u-(m-1)\le &R(\phi'(|x|/\lam_y))\py+O(\de),\\
\label{eqsl}		u\De u-(m-1)\le& R(|Du|)|Du|^2+O(\de)|Du|^2,\\
	\label{eqls}	u\De u-(m-1)\le& \frac{m-1}{n}|Du|^2+O(\de)|Du|^2,
	\end{align}with big O only depending on $m,n.$
\end{lem}
It turns out later the first line (\ref{eqfl}) is the most useful. (\ref{eqsl}) will not be used at all. (\ref{eqls}) will only be used once.
\begin{proof}
Apply Lemma \ref{lemuduu} and Lemma \ref{lemphip}. From 
\begin{align*}
	||Du|-|\phi'(|x|/\lam_y)||&=\left||Du|-|D_{x,z}\phi_{\lam_y}(|x|)|_{z=y}\right|\\&\le\left|D_{x,z}(u-\phi_{\lam_y}(|x|))|_{z=y}\right|\\&\le\de,
\end{align*}we deduce (\ref{eqfl})
 and for \ref{eqsl} 
	\begin{align*}
		u\De u-(m-1)\le&\frac{2(m-1)(\ai_0^2-\phi'(|x|/\lam_y)^2)}{2n\ai_0^2+(m-n-1)\phi'(|x|/\lam_y)}|Du|^2+O(\de)|Du|^2\\
		=&\frac{2(m-1)(\ai_0^2-|Du|^2+O(\de))}{2n\ai_0^2+(m-n-1)|Du|^2+O(\de)}|Du|^2+O(\de)|Du|^2\\
		=&\frac{2(m-1)(\ai_0^2-|Du|^2)}{2n\ai_0^2+(m-n-1)|Du|^2}|Du|^2+O(\de)|Du|^2,
	\end{align*}with big O only depending on $m,n.$
	(\ref{eqls}) follows directly.
\end{proof}
The special form of $R$  will be used to apply Cauchy Schwartz to (\ref{eqfl}) as follows
\begin{fact}\label{fctcs}
	For $a,b,a',b'\ge 0,$ we have
	\begin{align*}
	&\left(	a(\ai_0^2-\phi'(|x|/\lam_y)^2)+b\py
\right)\left(		a'(\ai_0^2-\phi'(|x|/\lam_y)^2)+b'\py\right)\\\ge&\left(\sqrt{a'b}+\sqrt{ab'}\right)^2(\ai_0^2-\phi'(|x|/\lam_y)^2)\py	\end{align*}
\end{fact}
\begin{proof}
	Note that $\ai_0^2-\phi'(|x|/\lam_y)^2,|Du|\ge0$. The rest follows directly from Cauchy-Schwarz.
\end{proof}
\section{Calculating $\di_g \frac{\na^g F}{\no{\na^g F}_g}$}\label{seccalc}
 We separately define the two branches of $E$.
\begin{defn}\label{defnfg}
Define for $(u,v)\in\R\times\R$
\begin{align*}
	F^+(u,v)=&\left(u^2-v^2\right)u^d,\\
	F^-(u,v)=&\left(u^2-v^2\right)v^d.
\end{align*}
We will use $F$ without superscript to denote both $F^+$ and $F^-.$
\end{defn}
Then we have
\begin{align*}
	E(x,\xi,y)=\begin{cases}
		F^+(u(x,y),|\xi|),&\textnormal{if }u(x,y)-|\xi|\ge 0,\\
		F^-(u(x,y),|\xi|),&\textnormal{if }u(x,y)-|\xi|\le 0.\\
	\end{cases}
\end{align*}
Let us calculate 	$\di_g \frac{\na^g F}{\no{\na^g F}_g}$. The following standard facts about $v=|\xi|,r=|x|,$ will be used several times without explaining.
\begin{fact}\label{fctxi}
	We have \begin{align*}
|Dv|^2=&1,	\De v=\frac{m-1}{v},D^2v(\na v,\na v)=0,\\
|Dr|^2=&1,	\De r=\frac{n-1}{r},D^2r(\na r,\na r)=0.
	\end{align*}
\end{fact}
\subsection{Calculating  $\frac{\na^g F^+}{\no{\na^g F^+}_g}$}
Let us first start with the gradient. We have
\begin{align*}
	D_iF=&F_uD_iu,
	D_\ai F=F_v D_\ai v.
\end{align*}
For
\begin{align*}
	F^+(u,v)=(u^2-v^2)u^d,
\end{align*}
we have
\begin{align*}
	F_u^+=&(d+2)u^{d+1}-du^{d-1}v^2=u^{d-1}\left((d+2)u^2-dv^2\right),\\
	F_v^+=&-2vu^d.
\end{align*}
In the metric $g=	dx^2+dy^2+fd\xi^2$, we have
\begin{align}
	\na^gF^+=&\sum_{i=1}^nD_iF\pd_{x_i}+\sum_{j=1}^lD_jF\pd_{y_j}+\frac{1}{f}\sum_{\ai=1}^mD_\ai F\pd_{\xi_\ai}\\
	=&F_u\na u+f\m F_v\na v\\
	=&u^{d-1}\left((d+2)u^2-dv^2\right)\na u-u^{d-1}f\m2vu\na v.\label{eqnae1}
\end{align}
Thus, we have
\begin{align*}
	\no{\na^gF^+}_g=&\sqrt{\sum_{i=1}^{n+l}D_iFD_iF+\frac{1}{f}\sum_{\ai=1}^mD_\ai FD_\ai F}=\sqrt{F_u^2|Du|^2+f\m F_v^2}\\
	=&u^{d-1}\sqrt{\left((d+2)u^2-dv^2\right)^2|Du|^2+f\m4u^2v^2}.
\end{align*}where we use $|D{v}|=1$ at the right-hand side.
This gives
\begin{align*}
	\frac{\na^gF^+}{\no{\na^gF^+}_g}=\frac{\left((d+2)u^2-dv^2\right)\na u-f\m2vu\na v}{\sqrt{\left((d+2)u^2-dv^2\right)^2|Du|^2+f\m4u^2v^2}}.
\end{align*}
\subsection{Calculating $\frac{\na^g F^-}{\no{\na^g F^-}_g}$}
For
\begin{align*}
	F(u,v)=(u^2-v^2)v^d,
\end{align*}
we have
\begin{align}
	F_u^-=&2uv^d,\\
	F_v^-=&dv^{d-1}u^2-(d+2)v^{d+1}=v^{d-1}\left(du^2-(d+2)v^2\right).\label{eqnae2}
\end{align}
Similar calculations as in the last subsection give
\begin{align*}
		\frac{\na^gF^-}{\no{\na^gF^-}_g}=\frac{2uv\na u+f\m\left(du^2-(d+2)v^2\right)\na v}{\sqrt{4u^2v^2|Du|^2+f\m\left(du^2-(d+2)v^2\right)^2}}
\end{align*}
To sum it up, we have shown that 
\begin{fact}\label{fctpq}
	We have \begin{align*}
		\frac{\na^gF}{\no{\na^gF}_g}=\frac{P(u,v)\na u+f\m Q(u,v)\na{v}}{\sqrt{(P(u,v))^2|Du|^2+f\m (Q(u,v))^2}},
	\end{align*}where $P$ and $Q$ are quadratic forms on $\R^2.$ And
	\begin{align*}
		P^+(u,v)=&
		(d+2)u^2-dv^2,Q^+(u,v)=
		-2uv,\\
		P^-(u,v)=&
		2uv,Q^-(u,v)=
		du^2-(d+2)v^2.
	\end{align*}
\end{fact}
\subsection{Calculating $\di_g \frac{\na^g F}{\no{\na^g F}_g}$}
\begin{lem}\label{lemdiv}
	We have
	\begin{align*}
		&f^{2}\left(P^2|Du|^2+f\m Q^2\right)^{\frac{3}{2}}\di_g	\frac{\na^gF}{\no{\na^gF}_g}
	\\=&\frac{m+1}{2}P(Q^2-P^2)Df\cdot Du+\left(Q^2P_u-PQ(P_v+Q_u)+P^2Q_v+P^3\frac{m-1}{u}+P^2Q\frac{m-1}{{v}}\right)f|Du|^2\\
	&+P\left(Q^2-P^2\right)f\De u+P^3\frac{m-1}{u}+Q^3\frac{m-1}{{v}}.
	\end{align*}
	Furthermore, if for any fixed $(x,y)$, we regard $f^{2}\left(P^2|Du|^2+f\m Q^2\right)^{\frac{3}{2}}\di_g	\frac{\na^gF}{\no{\na^gF}_g}	$ as a function of $(u,v,Du)$, we have
	\begin{align*}
		\left[f^{2}\left(P^2|Du|^2+f\m Q^2\right)^{\frac{3}{2}}\di_g	\frac{\na^gF}{\no{\na^gF}_g}	\right](u,u,Du)=0
	\end{align*}
\end{lem}
\begin{proof}
	Recall that $$\di_g=\frac{1}{\sqrt{\det}}\sum_{j=1}^{n+l}D_j\left(\sqrt{\det g}E^j\right)+\frac{1}{\sqrt{\det}}\sum_{\be=1}^{m}D_\be\left(\sqrt{\det g}E^\be\right),$$ with $\sqrt{\det g}=f^{\frac{m}{2}}$ . Thus,	We have
	\begin{align*}
		&f^{2}\left(P^2|Du|^2+f\m Q^2\right)^{\frac{3}{2}}\di_g	\frac{\na^gF}{\no{\na^gF}_g}\\
		=&f^{2-\frac{m}{2}}\left(P^2|Du|^2+f\m Q^2\right)^{\frac{3}{2}}	\Bigg(\sum_{j=1}^{n+l} D_j\bigg(\frac{f^{\frac{m}{2}}PD_ju}{\sqrt{P^2|Du|^2+f\m Q^2}}\bigg) +\sum_{\be=1}^m D_\be \bigg(\frac{f^{\frac{m}{2}-1}QD_\be{v}}{\sqrt{P^2|Du|^2+f\m Q^2}}\bigg) \Bigg)\\
		=&f^{2-\frac{m}{2}}\left(P^2|Du|^2+f\m Q^2\right)^{\frac{3}{2}}\cdot\\
		&\sum_{j=1}^{n+l}\left(\frac{D_j\left(f^{\frac{m}{2}}PD_ju\right)}{{\sqrt{P^2|Du|^2+f\m Q^2}}}-\frac{1}{2}\frac{f^{\frac{m}{2}}PD_ju}{{\sqrt{P^2|Du|^2+f\m Q^2}}^3}D_j\left({{P^2|Du|^2+f\m Q^2}}\right)\right)\\
		&+f^{2-\frac{m}{2}}\left(P^2|Du|^2+f\m Q^2\right)^{\frac{3}{2}}\cdot\\
		&\sum_{\be=1}^{m}\left(\frac{D_\be\left(f^{\frac{m}{2}-1}QD_\be{v}\right)}{{\sqrt{P^2|Du|^2+f\m Q^2}}}-\frac{1}{2}\frac{f^{\frac{m}{2}-1}QD_\be{v}}{{\sqrt{P^2|Du|^2+f\m Q^2}}^3}D_\be\left({{P^2|Du|^2+f\m Q^2}}\right)\right)\\
		=&f^{2-\frac{m}{2}}\left({P^2|Du|^2+f\m Q^2}\right)\left(\frac{m}{2}f^{\frac{m}{2}-1}P\left(Df\cdot Du\right)+f^{\frac{m}{2}}P_u|Du|^2+f^{\frac{m}{2}}P\De u\right)\\
		&-f^{2}P\left(PP_u|Du|^4+P^2D^2u\left(\na u,\na u\right)-\frac{1}{2}f^{-2}Q^2\left(Df\cdot Du\right)+f\m QQ_u|Du|^2\right)\\
		&+f\left({P^2|D u|^2+f\m Q^2}\right)\left(Q_v+Q\frac{m-1}{{v}}\right)\\&-fQ\left(PP_v|Du|^2+f\m QQ_v\right)\\
		=&\left(\frac{m}{2}fP^3|Du|^2+\frac{m+1}{2}PQ^2\right)\left(Df\cdot Du\right)+\left(Q^2P_u-PQ(P_v+Q_u)+P^2Q_v+P^2Q\frac{m-1}{{v}}\right)f|Du|^2\\
		&+\left({fP^3|D u|^2+PQ^2}\right)f\De u-f^2P^3D^2u\left(\na u,\na u\right)+Q^3\frac{m-1}{{v}}.
	\end{align*}
	Recall that the Euler-Lagrange equation for $u,f$ is \cite[Proof of Theorem 3.7 just before equation (3)]{LSfr}
	\begin{align*}
		\frac{1}{2}\left(m+(1+f|Du|^2)\m\right)Df\cdot Du=-f\De u+\frac{f^2}{1+f|Du|^2}D^2u(\na u,\na u)+\frac{m-1}{u}.
	\end{align*}
	Rearranging gives
	\begin{align}
		f^2D^2u(\na u,\na u)=&\frac{1}{2}\left(m(1+f|Du|^2)+1\right)Df\cdot Du+(1+f|Du|^2)f\De u-(1+f|Du|^2)\frac{m-1}{u}
	\end{align}
	Inserting into $f^{2}\left(P^2|Du|^2+f\m Q^2\right)^{\frac{3}{2}}\di_g	\frac{\na^gF}{\no{\na^gF}_g}$ above, we deduce that
	\begin{align*}
		&f^{2}\left(P^2|Du|^2+f\m Q^2\right)^{\frac{3}{2}}\di_g	\frac{\na^gF}{\no{\na^gF}_g}
		\\=&\frac{m+1}{2}P(Q^2-P^2)Df\cdot Du+\left(Q^2P_u-PQ(P_v+Q_u)+P^2Q_v+P^3\frac{m-1}{u}+P^2Q\frac{m-1}{{v}}\right)f|Du|^2\\
		&+P\left(Q^2-P^2\right)f\De u+P^3\frac{m-1}{u}+Q^3\frac{m-1}{{v}}.
	\end{align*}
		For every $(x,y)\in\R^n\times\R^l$, the hypersurface $M:|\xi|=u(x,y)$ is minimal \cite[Theorem 3.7]{LSfr}. In other words, the level set $F=0$ has mean curvature $0$. Since $\div_g 	\frac{\na^gF}{\no{\na^gF}_g}$ equals the mean curvature scalar of levels sets of $F,$ regarding $[f^{2}\left(P^2|Du|^2+f\m Q^2\right)^{\frac{3}{2}}\di_gE]$ as a function of $u,v,Du$, we deduce
	\begin{align*}\left[f^{2}\left(P^2|Du|^2+f\m Q^2\right)^{\frac{3}{2}}\di_g	\frac{\na^gF}{\no{\na^gF}_g}\right](u,u,Du)=0.
	\end{align*}
\end{proof}
\section{Factorization in $u\ge v$}\label{secfact+}
Let us prove (\ref{eqfact}). In this section we deal with the region $u\ge v,$ i.e., $E=F^+$. First we fix a notation
\begin{defn}
For a function $B$ with variables $(u,v)$	overline over $B,$ i.e., $\ov{B},$ means subtracting its value at $(u,u)$, i.e.,
	\begin{align*}
		\ov{B}(u,v)=B(u,v)-B(u,u).
	\end{align*}
\end{defn}
\subsection{Basic factorization lemma for $F^+$}
We will prove that
\begin{lem}\label{lemb+}
	We have 
	\begin{align*}
		uf^{2}\left(P^2|Du|^2+f\m Q^2\right)^{\frac{3}{2}}\di_g\frac{\na^g F}{\no{\na^g F}_g}=(u^2-v^2) B^+,
	\end{align*}with $B^+\ge 0$. 
\end{lem}
Let us first calculate the factorization exactly.
\begin{lem}\label{lemp+}
	For $(a,b)\in\R^2$ define the quadratic forms,
	\begin{align*}
		P_1^+(a,b)=& d^3 b^2+\left(d^3+6 d^2+12 d+8\right) a^2+\left(-2
		d^3-6 d^2-4 d\right) ab,\\
		P_2^+(a,b)=&b^2 \left(d^3 m-d^3\right)\\&+a^2 \left(d^3 m-d^3+4 d^2
		m-6 d^2+4 d m-12 d-8\right)+ab \left(-2 d^3
		m+2 d^3-4 d^2 m+2 d^2-4 d\right),\\
		P_3^+(a,b)=&d^3 b^2+\left(d^3+6 d^2+12 d+8\right) a^2+\left(-2
		d^3-6 d^2\right) ab
	\end{align*}
	We can write
	\begin{align*}
		&uf^{2}\left(P^2|Du|^2+f\m Q^2\right)^{\frac{3}{2}}\di_g\frac{\na^g F}{\no{\na^g F}_g}
		\\=&(u^2-v^2)\left(\frac{m+1}{2}P_1^+(u^2,v^2)(-uDf\cdot Du)+P_2^+(u^2,v^2)f|Du|^2+P_1^+(u^2,v^2)(-uf\De u)+(m-1)P_3^+(u^2,v^2)\right)
	\end{align*}
\end{lem}
	Set
\begin{align*}
	B^+=\frac{m+1}{2}P_1^+(u^2,v^2)(-uDf\cdot Du)+P_2^+(u^2,v^2)f|Du|^2+P_1^+(u^2,v^2)(-uf\De u)+(m-1)P_3^+(u^2,v^2).
\end{align*}
\begin{proof}
	By Lemma \ref{lemdiv}, we have
	\begin{align*}
		&[uf^{2}\left(P^2|Du|^2+f\m Q^2\right)^{\frac{3}{2}}\di_g\frac{\na^g F}{\no{\na^g F}_g}](u,v,Du)\\
		&[uf^{2}\left(P^2|Du|^2+f\m Q^2\right)^{\frac{3}{2}}\di_g\frac{\na^g F}{\no{\na^g F}_g}](u,v,Du)-[uf^{2}\left(P^2|Du|^2+f\m Q^2\right)^{\frac{3}{2}}\di_g\frac{\na^g F}{\no{\na^g F}_g}](u,u,Du)\\=&\frac{m+1}{2}\ov{uP(Q^2-P^2)}Df\cdot Du+\ov{u\left(Q^2P_u-PQ(P_v+Q_u)+P^2Q_v+P^3\frac{m-1}{u}+P^2Q\frac{m-1}{{v}}\right)}f|Du|^2\\
		&+\ov{uP\left(Q^2-P^2\right)}f\De u+\ov{u\left(P^3\frac{m-1}{u}+Q^3\frac{m-1}{{v}}\right)}.
	\end{align*}
	Recall Fact \ref{fctpq}. We have $P(u,v)=(d+2)u^2-dv^2$ and $Q(u,v)=-2vu$. This gives
	\begin{align*}
		P_u=&2(d+2)u,P_v=-2dv,Q_u=-2v,Q_v=-2u,\\
		\ov{uP(Q^2-P^2)}=&-u (u^2-v^2) \left(\left(d^3+6 d^2+12 d+8\right) u^4+\left(-2
		d^3-6 d^2-4 d\right) u^2 v^2+d^3 v^4\right),\\
		\ov{u\left(P^3\frac{m-1}{u}+Q^3\frac{m-1}{{v}}\right)}=&(m-1)(u^2-v^2)\left(d^3 v^4+\left(d^3+6 d^2+12 d+8\right) u^4+\left(-2
		d^3-6 d^2\right) u^2 v^2\right)
	\end{align*}
	and
	\begin{align*}
		&\ov{u\left(Q^2P_u-PQ(P_v+Q_u)+P^2Q_v+P^3\frac{m-1}{u}+P^2Q\frac{m-1}{{v}}\right)}\\
		=&\left(u^2-v^2\right) \bigg(v^4 \left(d^3 m-d^3\right)\\&+u^4 \left(d^3 m-d^3+4 d^2
		m-6 d^2+4 d m-12 d-8\right)+u^2 v^2 \left(-2 d^3
		m+2 d^3-4 d^2 m+2 d^2-4 d\right)\bigg),
	\end{align*}
	Mathematica verification will be provided at the end.
\end{proof}
\subsection{Positivity of $B^+$ on $\Om_K\cup\oi\cup \ot$}
Recall Fact \ref{fctfu}. We have
\begin{align*}
	&B^+\\
	=&P_1^+(u^2,v^2)O(\tau)+P_2^+(u^2,v^2)\ai_0^2(1+O(\tau))-(m-1)(1+O(\tau))P_1^+(u^2,v^2)+(m-1)P_3^+(u^2,v^2)\\
	=&u^4\left(-(m-1)P_1^+(1,t^+)+P_2^+(1,t^+)\ai_0^2+(m-1)P_3^+(1,t^+)+O(\tau)P_1^+(1,t^+)+O(\tau)P_2^+(1,t^+)\right),
	\end{align*}
	with $t^+=\frac{v^2}{u^2}\in[0,1].$
	
To prove Lemma \ref{lemb+}, by Fact  \ref{fctpoly}, it suffices to verify that 
\begin{align*}
V^+(t^+)=-(m-1)P_1^+(1,t^+)+P_2^+(1,t^+)\ai_0^2+(m-1)P_3^+(1,t^+)
\end{align*} has positive infimum for $t^+\in[0,1].$ 

For $(n,m)=(5,3),d=\frac{9}{5},$ direct calculation,
\begin{align*}
V^+(t^+)=	-(m-1)P_1^+(1,t^+)+P_2^+(1,t^+)\ai_0^2+(m-1)P_3^+(1,t^+)\ge\frac{8}{25},\forall t^+\in[0,1].
\end{align*}
For $(n,m)=(4,4),d=\frac{9}{5},$ direct calculation,
\begin{align*}
V^+(t^+)=	-(m-1)P_1^+(1,t^+)+P_2^+(1,t^+)\ai_0^2+(m-1)P_3^+(1,t^+)\ge\frac{16}{25},\forall t^+\in[0,1].
\end{align*}

For $n+m\ge 9,n\ge 3,m\ge 2,d=3$, we have
\begin{align*}
	V^+(0)=&\frac{25 (m-1) (3 m-5)}{n-1}>0,\\
	V^+(1)=&\frac{4 (m-1) (3 m+3 n-26)}{n-1}>0.
\end{align*}
As $V^+(t)$ is a quadratic polynomial, its minimum on $[0,1]$ is either achieved at endpoints or the axis of symmetry. The axis of symmetry of $V^+(t)$ is
\begin{align*}
	t_0=\frac{-15 m+2 n+8}{9-9 m}.
\end{align*}
If $t_0\in [0,1]$, then we have 
\begin{align*}
0\ge	-15 m+2 n+8\ge 9-9 m.
\end{align*}
On the other hand $n\ge 9-m,$ so we have
\begin{align*}
	-15 m+2 n+8\ge 26 - 17 m.
\end{align*}
Multiplying gives
\begin{align*}
	(	-15 m+2 n+8)^2\le (9-9m)(26-17m)
\end{align*}
We have
\begin{align*}
	V^+(t_0)=&\frac{60 m (n-6)-4 n (n+8)+311}{3 (n-1)}\\
	=&\frac{60 m (n-6)-4 n (n+8)+311+(-15 m+2 n+8)^2-(-15 m+2 n+8)^2}{3 (n-1)}\\
	\ge&\frac{60 m (n-6)-4 n (n+8)+311+(-15 m+2 n+8)^2-(9-9m)(26-17m)}{3 (n-1)}\\
	=&\frac{3 (m-1) (24 m-47)}{3(n-1)}\\
	>&0.
\end{align*}We are done.
\subsection{Positiveity of factors on $\oo$}
On $\oo,$ we have $f=1.$ Again, setting $t^+=\frac{v^2}{u^2}$, we have $
t^+\in[0,1],$ and
\begin{align*}
B^+=&P_2^+(u^2,v^2)|Du|^2+P_1^+(u^2,v^2)(-u\De u)+(m-1)P_3^+(u^2,v^2)\\=&u^4\left(P_2^+(1,t^+)|Du|^2+P_1^+(1,t^+)(-u\De u)+(m-1)P_3^+(1,t^+)\right).
\end{align*}
Note that $P_1(t)$ depends only on $d$ but not on $n,m$. 
\begin{align*}
	P_1^+(1,t^+)\ge\begin{cases}
		32&\textnormal{ for }d=3,\\
		\frac{112}{5}&\textnormal{ for }d=\frac{9}{5}.
	\end{cases}
\end{align*}
Mathematica verifications will be provided at the end. 

We have by Lemma \ref{lemudum}
\begin{align}
&P_2^+(1,t^+)|Du|^2+P_1^+(1,t^+)(-u\De u)+(m-1)P_3^+(1,t^+)\\=&P_2^+(1,t^+)|Du|^2-\left(u\De u-(m-1)\right)P_1^+(1,t^+)+(m-1)(P_3^+(1,t^+)-P_1^+(1,t^+)).
\end{align}
Set
\begin{align*}
	&L^+(t^+)=P_2^+(1,t^+)|Du|^2-\left(u\De u-(m-1)\right)P_1^+(1,t^+)+(m-1)(P_3^+(1,t^+)-P_1^+(1,t^+))
\end{align*}
Let $L^+_2,L^+_1,L^+_0$ be the coefficients of $(t^+)^2,(t^+)^1,(t^+)^0,$ respectively. Recall Fact \ref{fctl}. To prove Lemma \ref{lemb+}, it suffices to show that 
\begin{lem}\label{lemf+}
	\begin{enumerate}
		\item	$L^+_2+L^+_1+L^+_0\ge 0$, 
		\item  $L^+_0\ge 0$,
		\item $L^+_1+2L^+_0\ge 0.$
	\end{enumerate}
\end{lem}
The rest of this section will be devoted to the proof of Lemma \ref{lemf+}. We will split into subcases of $d=3,d=\frac{9}{5}$ and the three bullets of Lemma \ref{lemf+}.
\subsection{$d=3$}
We will verify the three conditions one by one.
\subsubsection{$d=3$, $L^+_2+L^+_1+L^+_0> 0$}\label{secl012}
By Lemma \ref{lemudum}, we have
\begin{align}\label{eql012}
\frac{1}{4}\left(	L_2^++L_1^++L_0^+\right)\ge&  -3 + 3 m+(-23 + 3 m - 8 R(\pyy))\py+O(\de)) \\
\ge&2.99(n-1)(\ai_0^2-\phi'(|x|/\lam_y)^2)+(2.99(m+n-1)-23)\py\\&-\frac{16(m-1)(\ai_0^2-\py)\py}{2n(\ai_0^2-\py)+(m+n-1)\py}\\&+0.01(m-1)+O(\de).
\end{align}
Using Fact \ref{fctcs}, we deduce that
\begin{align}
&\left(2.99(n-1)(\ai_0^2-\phi'(|x|/\lam_y)^2)+(2.99(m+n-1)-23)\py\right)\\&\cdot\left(2n(\ai_0^2-\py)+(m+n-1)\py\right)\\
\ge&\left(\sqrt{2.99(n-1)(m+n-1)}+\sqrt{2n\left(2.99(m+n-1)-23\right)}\right)^2(\ai_0^2-\py)\py.
\end{align} Set $N=(m+n)$. For any fixed $N$,
\begin{align}
	&\left(\sqrt{2.99(n-1)(m+n-1)}+\sqrt{2n\left(2.99(m+n-1)-23\right)}\right)^2- 16(m-1)\\
	=&\left(\sqrt{2.99(N-m-1)(N-1)}+\sqrt{2(N-m)\left(2.99(N-1)-23\right)}\right)^2- 16(m-1),
\end{align}
is a decreasing function in $m,$ thus achieving its minimum at $m=N-3$, i.e.,
\begin{align}
		&\left(\sqrt{2.99(n-1)(m+n-1)}+\sqrt{2n\left(2.99(m+n-1)-23\right)}\right)^2- 16(m-1)\\
		\ge&\left(\sqrt{5.98(N-1)}+\sqrt{6\left(2.99(N-1)-23\right)}\right)^2- 16(N-4)\\=&5.98N-5.98+17.94N-155.94-16N+64\\&+2\sqrt{5.98(N-1)6\left(2.99(N-1)-23\right)}\\=&7.92N-97.92+2\sqrt{5.98(N-1)6\left(2.99(N-1)-23\right)}\\
		\ge&7.92\cdot 9-97.92+2\sqrt{5.98(9-1)6\left(2.99(9-1)-23\right)}\\
		>&5.8\\
		>&0.\label{eql012f}
\end{align}
Combining through (\ref{eql012}) to (\ref{eql012f}), we deduce that
\begin{align*}
\frac{1}{4}\left(	L_2^++L_1^++L_0^+\right)>0.01(m-1)+O(\de)>0,
\end{align*}provided $\de$ small.
\subsubsection{$d=3$,$L_0^+\ge 0$}
By Lemma \ref{lemudum}, we have\begin{align*}
&\frac{n}{25}L_0^+\ge	(-5 n+ 3 mn - 5(m-1)+O(\de)) |Du^2|.
\end{align*}
We have
\begin{align*}
	&-5 n+ 3 mn - 5(m-1)\\
	=&3(m-2)(n-2)+(n+m-9)+2\\
	\ge&2.
\end{align*}and thus
\begin{align*}
	\frac{n}{25}L_0^+\ge (2+O(\de))|Du^2|\ge 0,
\end{align*}provided $\de$ small.
\subsubsection{$d=3$, $L_1^++2L_0^+\ge 0$}
By Lemma \ref{lemudum}, we have
\begin{align*}
\frac{1}{2}(L_1^++2L_0^+)\ge&-6 + 6 m + (-65 + 30 (m-1) - 65 R(\pyy))\py+O(\de)\\=&5.9(n-1)(\ai_0^2-\py)+\left(5.9(n-1)-65+30(m-1)\right)\py\\&-  \frac{130(m-1)(\ai_0^2-\py)\py}{2n(\ai_0^2-\py)+(m+n-1)\py}\\&+0.1(m-1)+O(\de).
	\end{align*}
Using $m+n-1\ge 8,2n\ge 6$, we have
\begin{align*}
&\left(	\sqrt{5.9(n-1)(m+n-1)}+\sqrt{2n\left(5.9(n-1)-65+30(m-1)\right)}\right)^2\\
\ge&6\cdot 5.9(n-1)+6\cdot 5.9(n-1)+180(m-1)-6\cdot 65\\
=&130(m-1)+50(m+n-2)-6\cdot 65+20.8(n-1)\\
\ge&130(m-1)+350-390+41.6\\
>&130(m-1).\end{align*}
Thus, by Fact \ref{fctcs}, we have
\begin{align*}
	\frac{1}{2}(L_1^++2L_0^+)\ge 0.1(m-1)+O(\de)>0,
\end{align*}provided $\de$ is small.
\subsection{$d=\frac{9}{5}=1.8$}
Now let us deal with the case of $d=\frac{9}{5}$ and $(n,m)=(4,4)$ or $(5,3)$.
\subsubsection{$d=1.8,L_2^++L_1^++L_0^+>0$}\label{secl01218}
We have
\begin{align}
\label{eql01218}L_2^++L_1^++L_0^+\ge&	-7.2 + 7.2 m + (-49.76 + 7.2 m - 22.4 R(\pyy))\py+O(\de)\\
\ge&7.1(n-1)(\ai_0^2-\py)+\left(7.1(m+n-8)+0.1m-0.06\right)\py\\&-\frac{44.8(m-1)(\ai_0^2-\py)\py}{2n(\ai_0^2-\py)+(m+n-1)\py}\\&+0.1(m-1)+O(\de).
\end{align}
Using $n+m=8,$ we have
\begin{align}
&\left(	\sqrt{7.1(n-1)(m+n-1)}+\sqrt{2n\left(0.1m-0.06\right)}
\right)^2\\
\ge&\label{eq621}\begin{cases}
192&\textnormal{ if }n=4,\\
	239&\textnormal{ if }n=5,
	\end{cases}\\>&44.8(m-1)\end{align}
Thus, we have
\begin{align*}
	L_2^++L_1^++L_0^+>0.1(m-1)+O(\de)>0,
\end{align*}provided $\de$ small.
\subsubsection{$d=1.8,$ $L_0^+\ge 0,$}
We have
\begin{align*}
L_0^+\ge&(-54.872 + 25.992 m - 54.872 \frac{m-1}{n}+O(\de))|Du^2|.
\end{align*}
When $(n,m)=(4,4),$ we have 
\begin{align*}
	L_0^+\ge&(7.942+O(\de))|Du^2|\ge 0,
\end{align*}provided $\de$ small.

When $(n,m)=(5,3),$ we have 
\begin{align*}
	L_0^+\ge&(1.1552+O(\de))|Du^2|\ge 0,
\end{align*}provided $\de$ small.
\subsection{$d=1.8$, $L_1^++2L_0^+>0,$}
We have
\begin{align*}
L_1^++2L_0^+\ge&	-7.2 + 7.2 m + (-98.8 + 27.36 m - 71.44 R(\pyy))\py +O(\de)\\
>&7.1(n-1)(\ai_0^2-\py)+(-98.8 + 27.36 m +7.1(n-1))\py\\&-\frac{142.88(m-1)(\ai_0^2-\py)\py}{2n(\ai_0^2-\py)+(m+n-1)\py}\\&+0.1(m-1)+O(\de)
 \end{align*}
We have
\begin{align*}
&	\left(\sqrt{7.1(n-1)(m+n-1)}+\sqrt{2n(-98.8 + 27.36 m +7.1(n-1))}\right)^2\\
	\ge&\begin{cases}
	794&\textnormal{ if }n=4\\
	620&\textnormal{ if }n=5.
	\end{cases}\\>& 142.88(m-1).
\end{align*}
Thus, we have
\begin{align*}
	L_1^++2L_0^+>0.1(m-1)+O(\de)>0,
\end{align*}provided $\de $ small.
\section{Factorization in $u\le v$}\label{secfact-}
In this section we deal with the region $u\le v,$ i.e., $E+F^-.$
\subsection{Factorization lemma of $F^-$}
We will prove that
\begin{lem}\label{lemb-}
	We have 
	\begin{align*}
		vf^{2}\left(P^2|Du|^2+f\m Q^2\right)^{\frac{3}{2}}\di_g\frac{\na^g F}{\no{\na^g F}_g}=(u^2-v^2) B^-,
	\end{align*}with $B^-\ge 0$. 
\end{lem}
Let us first calculate the factorization exactly.
\begin{lem}\label{lemp-}
	For $(a,b)\in\R^2$ define the quadratic forms,
	\begin{align*}
		P_1^-(a,b)=& 2b\left(d^2 a + (-4 - 4 d - d^2) b\right),\\
		P_2^-(a,b)=&(-2b)\left(a \left(d^2-2 d m+4 d\right)+\left(d^2+4 d+4\right) b\right),\\
		P_3^-(a,b)=&\left(d^3 a^2+\left(-2 d^3-6 d^2\right) a b+\left(d^3+6 d^2+12 d+8\right) b^2\right)
	\end{align*}
	We can write
	\begin{align*}
		&vf^{2}\left(P^2|Du|^2+f\m Q^2\right)^{\frac{3}{2}}\di_g\frac{\na^g F}{\no{\na^g F}_g}
		\\=&(u^2-v^2)\left(\frac{m+1}{2}P_1^-(u^2,v^2)(uDf\cdot Du)+P_2^-(u^2,v^2)f|Du|^2+P_1^-(u^2,v^2)(uf\De u)+(m-1)P_3^-(u^2,v^2)\right)
	\end{align*}
\end{lem}
Set
\begin{align*}
	B^-=\frac{m+1}{2}P_1^-(u^2,v^2)(uDf\cdot Du)+P_2^-(u^2,v^2)f|Du|^2+P_1^-(u^2,v^2)(uf\De u)+(m-1)P_3^-(u^2,v^2).
\end{align*}
\begin{proof}
	By Lemma \ref{lemdiv}, we have
	\begin{align*}
		&[vf^{2}\left(P^2|Du|^2+f\m Q^2\right)^{\frac{3}{2}}\di_g\frac{\na^g F}{\no{\na^g F}_g}](u,v,Du)\\
		&[vf^{2}\left(P^2|Du|^2+f\m Q^2\right)^{\frac{3}{2}}\di_g\frac{\na^g F}{\no{\na^g F}_g}](u,v,Du)-[vf^{2}\left(P^2|Du|^2+f\m Q^2\right)^{\frac{3}{2}}\di_g\frac{\na^g F}{\no{\na^g F}_g}](u,u,Du)\\=&\frac{m+1}{2}\ov{vP(Q^2-P^2)}Df\cdot Du+\ov{v\left(Q^2P_u-PQ(P_v+Q_u)+P^2Q_v+P^3\frac{m-1}{u}+P^2Q\frac{m-1}{{v}}\right)}f|Du|^2\\
		&+\ov{vP\left(Q^2-P^2\right)}f\De u+\ov{v\left(P^3\frac{m-1}{u}+Q^3\frac{m-1}{{v}}\right)}.
	\end{align*}
	Recall Fact \ref{fctpq}.  We have  $P(u,v)=2uv$ and $Q(u,v)=du^2-(d+2)v^2$. This implies.
	\begin{align*}
		P_u=&2v,P_v=2u,Q_u=2du,Q_v=-2(d+2)v,\\
		\ov{vP(Q^2-P^2)}=&(u^2-v^2)2uv^2\left(d^2 u^2 + (-4 - 4 d - d^2) v^2\right),\\
		\ov{v\left(P^3\frac{m-1}{u}+Q^3\frac{m-1}{{v}}\right)}=&(m-1)(u^2-v^2)\left(d^3 u^4+\left(-2 d^3-6 d^2\right) u^2 v^2+\left(d^3+6 d^2+12 d+8\right) v^4\right)
	\end{align*}
	and
	\begin{align*}
		&\ov{v\left(Q^2P_u-PQ(P_v+Q_u)+P^2Q_v+P^3\frac{m-1}{u}+P^2Q\frac{m-1}{{v}}\right)}\\=&(u^2-v^2)(-2v^2)\left(u^2 \left(d^2-2 d m+4 d\right)+\left(d^2+4 d+4\right) v^2\right).	\end{align*}
	Mathematica verification will be provided at the end.
\end{proof}

\subsection{Positivity of $B^-$ on $\Om_K\cup\oi\cup \ot$}
Recall Fact \ref{fctfu}. We have
\begin{align*}
	&B^-\\
	=&P_1^-(u^2,v^2)O(\tau)+P_2^-(u^2,v^2)\ai_0^2(1+O(\tau))+(m-1)(1+O(\tau))P_1^-(u^2,v^2)+(m-1)P_3^-(u^2,v^2)\\
	=&v^4\left((m-1)P_1^-(t^-,1)+P_2^-(t^-,1)\ai_0^2+(m-1)P_3^-(t^-,1)+O(\tau)P_1^-(t^-,1)+O(\tau)P_2^-(t^-,1)\right),
\end{align*}
with $t^-=\frac{u^2}{v^2}\in[0,1].$

To prove Lemma \ref{lemb+}, by Fact  \ref{fctpoly}, it suffices to verify that 
\begin{align*}
	V^-(t^-)=(m-1)P_1^-(t^-,1)+P_2^-(t^-,1)\ai_0^2+(m-1)P_3^-(t^-,1)
\end{align*} has positive infimum for $t^-\in[0,1].$ 

For $(n,m)=(5,3),d=\frac{9}{5},$ direct calculation,
\begin{align*}
	V^-(t^-)=	(m-1)P_1^-(t^-,1)+P_2^-(t^-,1)\ai_0^2+(m-1)P_3^-(t^-,1)\ge\frac{8}{25},\forall t^-\in[0,1].
\end{align*}
For $(n,m)=(4,4),d=\frac{9}{5},$ direct calculation,
\begin{align*}
	V^-(t^-)=	(m-1)P_1^-(t^-,1)+P_2^-(t^-,1)\ai_0^2+(m-1)P_3^-(t^-,1)\ge\frac{16}{25},\forall t^-\in[0,1].
\end{align*}

For $n+m\ge 9,n\ge 3,m\ge 2,d=3$, we have
\begin{align*}
	V^-(0)=&\frac{25 (m-1) (3 n-5)}{n-1}>0,\\
	V^-(1)=&\frac{4 (m-1) (3 m+3 n-26)}{n-1}>0.
\end{align*}
As $V^-(t)$ is a quadratic polynomial, its minimum on $[0,1]$ is either achieved at endpoints or the axis of symmetry. The axis of symmetry of $V^-(t)$ is
\begin{align*}
	t_0=\frac{2 m-15 n+8}{9-9 n}.
\end{align*}
If $t_0\in [0,1]$, then we have 
\begin{align*}
	0\ge 2 m-15 n+8\ge9-9 n.
\end{align*}
On the other hand $m\ge 9-n,$ so we have
\begin{align*}
	 2 m-15 n+8\ge 26 - 17 n.
\end{align*}
Multiplying gives
\begin{align*}
	(2 m-15 n+8)^2\le (9-9n)(26-17n)
\end{align*}
We have
\begin{align*}
	V^-(t_0)=&\frac{(m-1) (-4 m (m-15 n+8)-360 n+311)}{3 (n-1)^2}\\
	=&\frac{(m-1) (-4 m (m-15 n+8)-360 n+311+(2 m-15 n+8)^2-(2 m-15 n+8)^2)}{3 (n-1)^2}\\
	\ge&\frac{(m-1) (-4 m (m-15 n+8)-360 n+311+(2 m-15 n+8)^2-(9-9n)(26-17n))}{3 (n-1)^2}\\
	=&\frac{(m-1)3 (n-1) (24 n-47)}{3(n-1)^2}\\
	>&0.
\end{align*}We are done.
\subsection{Positiveity of $B^-$ on $\oo$}
On $\oo,$ we have $f=1$. Set $$t^-=\frac{u^2}{v^2}\in[0,1],$$ so we have
\begin{align*}
	B^-=&P_2^-(u^2,v^2)|Du|^2+P_1^-(u^2,v^2)(u\De u)+(m-1)P_3^-(u^2,v^2)\\
	=&v^4\left(P_2^-(t^-,1)|Du|^2+P_1^-(t^-,1)\left(u\De u-(m-1)\right)+(m-1)(P_3^-(t^-,1)+P_1^-(t^-,1))\right)
\end{align*}
Note that $P_1(t^-,1)<0$ always holds. 
Set
\begin{align*}
	&L^-(t^-)=P_2^-(t^-,1)|Du|^2+P_1^-(t^-,1)\left(u\De u-(m-1)\right)+(m-1)(P_3^-(t^-,1)+P_1^-(t^-,1))
\end{align*}
Let $L^-_2,L^-_1,L^-_0$ be the coefficients of $(t^-)^2,(t^-)^1,(t^-)^0,$ respectively. Recall Fact \ref{fctl}. To prove Lemma \ref{lemb-}, it suffices to show that 
\begin{lem}\label{lemf-}
	\begin{enumerate}
		\item	$L^-_2+L^-_1+L^-_0\ge 0$, 
		\item  $L^-_0\ge 0$,
		\item $L^-_1+2L^-_0\ge 0.$
	\end{enumerate}
\end{lem}
The rest of this section will be devoted to the proof of Lemma \ref{lemf-}. We will split into subcases of $d=3,d=\frac{9}{5}$ and the three bullets of Lemma \ref{lemf-}.
\subsection{$d=3$}
In this subsection we deal with $d=3$, $n\ge 3,m\ge 2,n+m\ge 9$.
\subsubsection{$d=3$, $L^-_2+L^-_1+L^-_0\ge 0$}
We have
\begin{align}\label{eql012-}
\frac{1}{4}\left(L^-_2+L^-_1+L^-_0\right)\ge&
-3 + 3 m + (-23 + 3 m - 8R(\pyy) ) \py+O(\de).\end{align}Note that except for possible differences in the $O(\de)$ term, (\ref{eql012-}) is exactly the same as (\ref{eql012}). Thus, the same reasoning as in Section \ref{secl012} works and we have
\begin{align*}
	\frac{1}{4}\left(L^-_2+L^-_1+L^-_0\right)>0.01(m-1)+O(\de)>0,
\end{align*}provided $\de$ small.
\subsubsection{$d=3$, $L_0^-\ge 0$}
We have
\begin{align*}
\frac{1}{25}L_0^-\ge&	-3 + 3 m + \left(-2 - 2 R(\pyy)\right) \py+O(\de)\\
\ge&2.9(n-1)(\ai_0^2-\py)+\left(2.9n-4.9\right)\py\\
&-\frac{4(m-1)(\ai_0^2-\py)\py}{2n(\ai_0^2-\py)+(m+n-1)\py}\\&+0.1(m-1)+O(\de).
\end{align*}
We have
\begin{align*}
	&\left(\sqrt{2.9(n-1)(m+n-1)}+\sqrt{(2.9n-4.9)2n}\right)^2\\
	\ge&2.9(n-1)(m+n-1)\\
	\ge&5.8(m-1)\\
	>&4(m-1).
\end{align*}
Thus, we have
\begin{align*}
	\frac{1}{25}L_0^->0.1(m-1)+O(\de)>0,
\end{align*}provided $\de$ small.
\subsubsection{$d=3$, $L_1^-+2L_0^-\ge 0$}
We have
\begin{align*}
\frac{1}{2}\left(	L_1^-+2L_0^-\right)\ge&-30 + 30 m + (-71 + 6 m - 41 R(\pyy)) \py+O(\de)\\
\ge&29.9(n-1)(\ai_0^2-\py)+\left(-71+6m+29.9(n-1)\right)\py\\&-\frac{82(m-1)(\ai_0^2-\py)\py}{2n(\ai_0^2-\py)+(m+n-1)\py}\\&+0.1(m-1)+O(\de).
\end{align*}
By $n\ge 3,$ we have
\begin{align*}
&\left(\sqrt{29.9(n-1)(m+n-1)}+\sqrt{\left(-71+6m+29.9(n-1)\right)2n}\right)	^2\\
\ge&59.8(m-1+2)+36(m-1)+36+6(59.8-71)\\
>&82(m-1).
\end{align*}
Thus, we have
\begin{align*}
	\frac{1}{2}\left(	L_1^-+2L_0^-\right)>0.1(m-1)+O(\de)>0,
\end{align*}provided $\de$ small.
\subsection{$d=\frac{9}{5}=1.8$}
Now we deal with $d=3$, $(n,m)=(4,4)$ or $(5,3)$. 
\subsubsection{$d=1.8,$ $L^-_2+L^-_1+L^-_0\ge0$}\label{secl21018-}We have
\begin{align}\label{eql01218-}
	L^-_2+L^-_1+L^-_0\ge&-7.2 + 7.2 m + (-49.76 + 7.2 m - 22.4 R(\pyy)) \py+O(\de)
\end{align}Again, (\ref{eql01218-}) is the same as (\ref{eql01218}) barring possible differences in $O(\de)$. The same reasoning as Section \ref{secl01218} applies and we deduce that
\begin{align*}
		L^-_2+L^-_1+L^-_0>0.1(m-1)+O(\de)>0,
\end{align*}provided $\de$ small.
\subsubsection{$d=1.8$, $L_0^-\ge 0$}\label{secl0-}
We have
\begin{align*}
	L_0^-\ge&-25.992 + 25.992 m + (-28.88 - 28.88 R(\pyy))\py+O(\de)\\
	\ge&25.99(n-1)(\ai_0^2-\py)+(-28.88+25.99(n-1))\py\\&-\frac{57.76(m-1)(\ai_0^2-\py)\py}{2n(\ai_0^2-\py)+(m+n-1)\py}\\&+0.002(m-1)+O(\de).
\end{align*}
We have
\begin{align*}
	&\left(\sqrt{25.99(n-1)(m+n-1)}+\sqrt{2n(-28.88+25.99(n-1))}\right)^2\\
	\ge&\begin{cases}
		1864,\textnormal{ if }n=4,\\
		2956,\textnormal{ if }n=5,
	\end{cases}\\\ge&57.76(m-1).
\end{align*}
Thus, we have
\begin{align*}
	L_0^->0.002(m-1)+O(\de)>0,
\end{align*}provided $\de$ small.
\subsection{$d=1.8$, $L_1^-+2L_0^-\ge 0$}\label{secl10-}
We have
\begin{align*}
	L_1^-+2L_0^-\ge&-27.36 + 27.36 m + (-78.64 + 7.2 m - 51.28 R(\pyy)) \py+O(\de)\\
	\ge&27.3(n-1)(\ai_0^2-\py)+(27.3(n-1)-78.64+7.2m)\py\\
	&-\frac{102.56(m-1)(\ai_0^2-\py)\py}{2n(\ai_0^2-\py)+(m+n-1)\py}\\&+0.06(m-1)+O(\de).
\end{align*}
We have
\begin{align*}
&	\left(\sqrt{27.3(n-1)(m+n-1)}+\sqrt{2n(27.3(n-1)-78.64+7.2m)}\right)^2	\\\ge&\begin{cases}
		1596,\textnormal{ if }n=4,\\
		2548,\textnormal{ if }n=5,
	\end{cases}\\\ge&102.56(m-1).
\end{align*}
\section{Wrap up the proof}\label{secwrap}
Lemma \ref{lemb+} and Lemma \ref{lemb-} together prove Lemma \ref{lemsub}, giving the first and last bullet of Fact \ref{fctsubc}. To verify the second and the third bullet of Fact \ref{fctsubc} we need to check the regularity of $\frac{\na^g E}{\no{\na^gE}_g}$. Set 
\begin{align*}
	\sing\frac{\na^g E}{\no{\na^gE}_g}=&\{(x,\xi,y)|\xi=0\}\cup\{(x,\xi,y)|x=0\},\\\sing^{Lip}\frac{\na^g E}{\no{\na^gE}_g}=&	\sing\frac{\na^g E}{\no{\na^gE}_g}\cup M.
\end{align*}Let us first verify that the set of discontinuity of the vector field $\frac{\na^g E}{\no{\na^gE}_g}$ is contained in $	\sing\frac{\na^g E}{\no{\na^gE}_g}$.The discontinuity can happen in three instances,
\begin{itemize}
	\item $\na^gE=0,$
	\item $\na^gE\not=0$ and $\na^gE$ is not continuous.
	\item on $M$ where $F^+$ and $F^-$ patch together.
\end{itemize}
For the first case, note that $\na v$ never vanishes, except at $\xi=0,$ where it is undefined. Thus, by (\ref{eqnae1}), for $u\ge v,$ $\na^g E=0$ only when $uv=0,$ i.e., $u=0$ or $v=0.$ As $u\ge v\ge0,$ $u=0$ forces $v=0.$ On the other hand, by (\ref{eqnae2}), when $v\ge u,$ $\na^gE=0$ only when $du^2-(d+2)v^2,$ which only happens when $u=v=0.$ To sum it up, $\na^gE(x,\xi,y)=0$ implies $v=0.$

For the second case, $\na^g E$ is nonzero but not continuous. This happens when $u,\na u$ or $\na v$ is not continuous, i.e., on $\sing M\cup v\m(0)\s \{x=0\}\cup\{\xi=0\}.$

For the last case, note that $\frac{\na^g E}{\no{\na^gE}_g}$ is continuous across $M\setminus\sing\frac{\na^g E}{\no{\na^gE}_g},$ as
\begin{align*}
	&P^+=P^-,Q^+=Q^-,
\end{align*} on $M.$ 
 Since $\{x=0\}\cup\{\xi=0\}$ has Hausdorff dimension at most $\max\{n+l,m+l\}$, we are done verifying the second bullet of Fact \ref{fctsubc}. 

As $M$ is a stable stationary varifold, by the monotonicity formula \cite{WA}, the Hausdorff $(n+m+l-1)$-dimensional measure restricted to $M$ is $\si$-finite, so does the restriction to $M\cup\sing\frac{\na^g E}{\no{\na^gE}_g}.$
The third bullet of Fact \ref{fctsubc} follows directly as $\frac{\na^g F^+}{\no{\na^gF^-}_g}$ and $\frac{\na^g F^-}{\no{\na^gF^-}_g}$ are locally Lipschitz in the complement of $\sing^{Lip}\frac{\na^g E}{\no{\na^gE}_g}$.
\section{Discussions}\label{secconc}
\subsection{One-sided area-minimizing}
Simon's construction covers $n\ge 3,m\ge 2,n+m\ge 8.$ 

In these ranges, every Lawson cone $M_{(n,m)}$ is area-minimizing, except for $M_{(6,2)}$ \cite{BLep,PSmc}.
\begin{fact}
	\cite{FLms} The Lawson cone $M_{(6,2)}$ and $M_{(2,6)}$ are one-sided area-minimizing.
\end{fact}
Using the same estimates in this manuscript, we have
\begin{thm}\label{thmo}
	The hypersurface $M\s\R^{6+2+l}$ in \cite[Theorem 3.7]{LSfr} is area-minimizing in the set $$\{(x,\xi,y)|v(x,\xi,y)-u(x,\xi,y)\ge0\},$$ provided the construction parameters $\tau$ and $\de$ are small. 
\end{thm} 
\begin{proof}
Use the same notation as in Section \ref{secfact-}. Set $d=\frac{9}{5},(n,m)=(6,2).$ We only need to prove that 
\begin{align*}V^-(t^-)&>0,\forall t^-\in[0,1],\\
	L^-_2+L_1^-+L_0^-&\ge 0,L_0^-\ge 0,L_1^-+2L_0^-\ge 0.
\end{align*}Direct calculation gives,
\begin{align*}
V^-(t^-)=	(m-1)P_1^-(t^-,1)+P_2^-(t^-,1)\ai_0^2+(m-1)P_3^-(t^-,1)\ge\frac{16}{125},\forall t^-\in[0,1].
\end{align*}
By the same symbolic calculations as in Sections \ref{secl21018-} and \ref{secl01218},  $L^-_2+L_1^-+L_0^-\ge0,$ follows from 
\begin{align*}
	\left(	\sqrt{7.1(n-1)(m+n-1)}+\sqrt{2n\left(0.1m-0.06\right)}
	\right)^2
	\ge282>44.8(m-1)\end{align*}

By the same symbolic calculations as in Section \ref{secl0-}, $L_0^-\ge 0$ follows from 
\begin{align*}
\left(\sqrt{25.99(n-1)(m+n-1)}+\sqrt{2n(-28.88+25.99(n-1))}\right)^2	\ge 4223>57.76(m-1).
\end{align*}

By the same symbolic calculations as in Section \ref{secl10-}, $L_1^-+2L_0^-$ follows from
\begin{align*}
	\left(\sqrt{27.3(n-1)(m+n-1)}+\sqrt{2n(27.3(n-1)-78.64+7.2m)}\right)^2\ge 3643>102.56(m-1).
\end{align*}
Then the same reasoning as in Section \ref{secwrap} finishes the proof.
\end{proof}\subsection{The choice of $E$}
Besides Definition \ref{defnh}, the author has tried $E$ being quartic polynomials in $u,v$. In the case of the original Lawson cones, both types of $E$ provide subcalibrations via $\frac{\na E}{|\na E|}$. 

In particular, when $m=n,$ using $E=u^4-v^4$ as in \cite{GDEPss} lead to significantly shorter proof for $m,n$ sufficiently large. Unfortunately, they do not work for all cases of Lawson cones. On the other hand, general quartic polynomials do cover all Lawson cones, e.g., \cite{CHsub} and the author suspect they might work for Simon's hypersurfaces as well. However, the calculations are significantly more complicated as we will need to calculate infimum of polynomials of order $9$ in general.

The author choose to use the subcalibrations in \cite{ZLlc}, which not only works for almost all cases but also only involves optimization of quadratic polynomials. Unfortunately, even though the proof is conceptually simple, the successive quadratic minimum calculations make the paper exhausting to read and write.
\subsection{The case of $M_{(3,5)}$}
The author believes that Simon's hypersurfaces based on the Lawson $M_{(3,5)}$ is also area-minimizing, but could not prove so. The main reason is that the subcalibration profie we use did not have a positive enough margin in the quadratic form factor in (\ref{eqfact}). Neither do quartic polynomials in $u,v.$
\subsection{Open cases of Almgren's conjecture}
Combining \cite{LSfr} and \cite{ZLa}, Almgren's conjecture \ref{ca} are settled in Riemannian manifolds in the following dimension and codimension range
\begin{itemize}
	\item codimension $=1$,
	\item codimension $=2$, dimension $\ge 8,$
	\item dimension $\ge 3$, codimension $\ge 3.$
\end{itemize} 
The second bullet follows via embedding Simon's hypersurface into $\R^{n+m+l}\times\R$. The remaining cases are
\begin{itemize}
	\item whether fractal singular sets can happen on analytic, in particular flat Euclidean, metrics,
	\item codimension $=2,$ dimension $=3,4,5,6,7$.
\end{itemize}
		\printbibliography
		\appendix
		We collect mathematical verifications. They are in the order of verifications for Sections \ref{seces}, \ref{secfact+} and \ref{secfact-}.\includepdfmerge[pages=-, nup=2x2, frame=true]{verificationsphi.pdf,verificationsFpositive.pdf,verificationsFnegative.pdf}
\end{document}